\documentclass[titlepage,12pt]{article} 
\usepackage{hyperref}
\usepackage[usenames,dvipsnames]{xcolor}
\usepackage{amssymb,amsthm,amsmath} 
\usepackage{mathrsfs}
\usepackage[a4paper]{geometry}
\usepackage{datetime2}
\usepackage[utf8]{inputenc}
\usepackage[italian,english]{babel}

\date{}

\newcommand{\ep}{\varepsilon}
\newcommand{\re}{\mathbb{R}}

\newcommand{\LS}{\mathcal{LS}}

\newtheorem{thm}{Theorem}[section]
\newtheorem{thmbibl}{Theorem}

\newtheorem{rmk}[thm]{Remark}
\newtheorem{prop}[thm]{Proposition}
\newtheorem{defn}[thm]{Definition}
\newtheorem{cor}[thm]{Corollary}

\newtheorem{lemma}[thm]{Lemma}

\title{Almost global existence for Kirchhoff equations around global solutions}

\author{Marina Ghisi\vspace{1ex}\\ 
{\normalsize Università degli Studi di Pisa} \\
{\normalsize Dipartimento di Matematica}\\ 
{\normalsize PISA (Italy)}\\
{\normalsize e-mail: \texttt{marina.ghisi@unipi.it}}
\and
Massimo Gobbino\vspace{1ex}\\ 
{\normalsize Università degli Studi di Pisa} \\
{\normalsize Dipartimento di Matematica}\\ 
{\normalsize PISA (Italy)}\\  
{\normalsize e-mail: \texttt{massimo.gobbino@unipi.it}}
}

\begin{document}
\maketitle

\begin{abstract}

It is well-known that the life span of solutions to Kirchhoff equations tends to infinity when initial data tend to zero. These results are usually referred to as \emph{almost global existence}, at least in a neighborhood of the null solution. 

Here we extend this result by showing that the life span of solutions is lower semicontinuous, and in particular it tends to infinity whenever initial data tend to some limiting datum that originates a global solution. 

We also provide an estimate from below for the life span of solutions when initial data are close to some of the classes of data for which global existence is known, namely data with finitely many Fourier modes, analytic data and quasi-analytic data.

\vspace{6ex}

\noindent{\bf Mathematics Subject Classification 2020 (MSC2020):} 
35L90, 35L20, 35L72.

		
\vspace{6ex}


\end{abstract}

 
\section{Introduction}

Let $H$ be a real Hilbert space, and let $A$ be a positive self-adjoint operator on $H$ with dense domain $D(A)$. In this paper we consider the abstract evolution equation 
\begin{equation}
u''(t)+m\left(|A^{1/2}u(t)|^{2}\right)Au(t)=0,
\label{K:eqn}
\end{equation}
where $m:[0,+\infty)\to\re$ is a nonlinearity that we always assume to be of class $C^{1}$ and to satisfy the strict hyperbolicity assumption
\begin{equation}
	m(\sigma)\geq\nu_{0}>0
	\quad\quad
	\forall\sigma\geq 0.
	\label{hp:m-sh}
\end{equation}

Equation (\ref{K:eqn}) is an abstract version of the hyperbolic partial differential equation introduced by  G.~Kirchhoff in the celebrated monograph~\cite[Section~29.7]{Kirchhoff} as a model for the small transversal vibrations of elastic strings or membranes.

\paragraph{\textmd{\textit{Local and global existence results}}}

Existence of local/global solutions to equation (\ref{K:eqn}) with initial data
\begin{equation}
u(0)=u_{0},
\qquad\qquad
u'(0)=u_{1},
\label{K:data}
\end{equation}
has been deeply investigated in the literature. 

Concerning local-in-time existence, the classical result is that, for every initial condition $(u_{0},u_{1})\in D(A^{3/4})\times D(A^{1/4})$, there exists a positive real number $T$ such that problem (\ref{K:eqn})--(\ref{K:data}) admits a solution in the space
\begin{equation}
C^{0}\left([0,T],D(A^{3/4})\right)\cap C^{1}\left([0,T],D(A^{1/4})\right),
\label{defn:strong-sol}
\end{equation}
and this solution is unique among solutions in the same space. This result was substantially established by S.~Bernstein in the pioneering paper~\cite{1940-Bernstein}, and then refined by many authors (see~\cite{1996-TAMS-AroPan} for a modern version). Existence and uniqueness of local solutions behind this regularity threshold, and in particular for initial data in the so-called energy space $D(A^{1/2})\times H$, is still a largely open problem (the only result we are aware of in this direction is contained in the recent paper~\cite{gg:EnergySpaceSG}).

Global-in-time solutions to problem (\ref{K:eqn})--(\ref{K:data}) are known to exist in many different special cases, involving either special initial data (such as analytic data \cite{1940-Bernstein,1984-RNM-AroSpa,1992-InvM-DAnSpa,1992-Ferrara-DAnSpa}, quasi-analytic data \cite{1984-Tokyo-Nishihara,gg:K-Nishihara}, lacunary data \cite{2005-JDE-Manfrin,2006-JDE-Hirosawa,gg:K-Manfrinosawa,gg:K-Nishihara,2015-NLATMA-Hirosawa}), or special nonlinearities~\cite{1985-Pohozaev}, or dispersive equations and small data \cite{1980-QAM-GreHu,1993-ARMA-DAnSpa,2005-JDE-Yamazaki,2013-JMPA-MatRuz}, or spectral gap operators \cite{gg:K-Nishihara}.

There are no examples of local-in-time solutions that blow-up in some sense in finite time. As far as we know, it might happen that problem (\ref{K:eqn})--(\ref{K:data}) admits a global solution for every initial datum $(u_{0},u_{1})\in D(A^{3/4})\times D(A^{1/4})$, and even for every initial datum $(u_{0},u_{1})\in D(A^{1/2})\times H$. This remains the main notorious open question for Kirchhoff equations.

\paragraph{\textmd{\textit{Almost global existence}}}

Let us consider equation (\ref{K:eqn}) with initial data
\begin{equation}
u(0)=\ep u_{0},
\qquad\qquad
u'(0)=\ep u_{1},
\nonumber
\end{equation}
where $(u_{0},u_{1})\in D(A^{3/4})\times D(A^{1/4})$ and $\ep$ is a positive real number. In this case it turns out that the solution is defined at least in some interval $[0,T_{\ep}]$, where $T_{\ep}\to +\infty$ as $\ep\to 0^{+}$. In other words, the life span of the solution tends to $+\infty$ when initial data tend to~0. Results of this type are typical for nonlinear partial differential equations, and they are usually referred to as ``almost global existence''.

The simpler result of almost global existence is a by-product of the classical local existence result. For example, the local existence result of~\cite{1996-TAMS-AroPan} implies that, if
\begin{equation}
|u_{1}|^{2}+|A^{1/4}u_{1}|^{2}+|A^{1/2}u_{0}|^{2}+|A^{3/4}u_{0}|^{2}\leq\ep^{2},
\label{hp:smallness}
\end{equation}
then the solution to problem (\ref{K:eqn})--(\ref{K:data}) is defined at least on an interval $[0,T_{\ep}]$ with
\begin{equation}
T_{\ep}\geq\frac{c_{0}}{\ep^{2}},
\label{th:ep2}
\end{equation}
where $c_{0}$ is a suitable constant that depends on the behavior of $m$ in a neighborhood of the origin.

More refined results have been proved in the last years, with the aim of showing that the life span of solutions is longer than the one prescribed by (\ref{th:ep2}). In this direction, P.~Baldi and E.~Haus~\cite{2020-Nonlinearity-BalHau} considered (\ref{K:eqn}) in the concrete case where $m(\sigma):=1+\sigma$, and $A$ is (minus) the Laplacian in $[0,2\pi]^{d}$ with periodic boundary conditions. They obtained an estimate of the form
\begin{equation}
T_{\ep}\geq\frac{c_{1}}{\ep^{4}}
\label{th:ep4}
\end{equation}
provided that initial data satisfy the smallness condition (\ref{hp:smallness}) in $D(A^{3/4})\times D(A^{1/4})$ in dimension $d=1$, and the analogous smallness condition in $D(A)\times D(A^{1/2})$ in dimension $d\geq 2$. Their proof relies on the theory of normal forms for Hamiltonian systems, and the stronger result in dimension $d=1$ depends on the fact that the difference between any two different eigenvalues of the operator $A$ is in that case bounded from below by a positive constant, and this guarantees a better cancelation of non-resonant modes.

The same estimate (\ref{th:ep4}) had already been obtained independently by R.~Manfrin in the early 2000s. Manfrin's result, that was reported in some conference but never published, was proved by a completely different technique, and it was valid for the abstract equation (\ref{K:eqn}) with $m$ of class $C^{2}$, but again with the smallness condition in $D(A)\times D(A^{1/2})$. As far as we know, it is not clear how to extend Manfrin's technique to less regular data, even under suitable assumptions on the eigenvalues of $A$, and therefore the one dimensional case by Baldi and Haus is for the time being the only example of estimate (\ref{th:ep4}) below the regularity threshold of $D(A)\times D(A^{1/2})$.

More recently, P.~Baldi and E.~Haus~\cite{2022-SIAM-BalHau} were able to obtain the even better estimate
\begin{equation}
T_{\ep}\geq\frac{c_{2}}{\ep^{6}},
\nonumber
\end{equation}
again by relying on the theory of normal forms, but in this case the restrictions on initial data are rather severe.

\paragraph{\textmd{\textit{Our contribution}}}

The question that we asked ourselves was the following.
\begin{quote}
Assume that $\{(u_{0\ep},u_{1\ep})\}$ is a family of initial data that converges in some sense to some limiting initial datum $(u_{0},u_{1})$, and assume that equation (\ref{K:eqn}) admits a global-in-time solution with initial datum $(u_{0},u_{1})$. Can we conclude that the life span $T_{\ep}$ of solutions with initial data $(u_{0\ep},u_{1\ep})$ tends to $+\infty$ as $\ep\to 0^{+}$?
\end{quote}

In this paper we give a positive answer to this question.
\begin{itemize}

\item  From the qualitative point of view, in Theorem~\ref{thm:lsc} we prove more generally that the life span of solutions (which is either a positive real number or $+\infty$) is lower semicontinuous with respect to convergence of initial data in $D(A^{3/4})\times D(A^{1/4})$, and that solutions depend continuously on initial data in the longest possible time interval. In particular, if for some reason the limiting initial datum $(u_{0},u_{1})$ originates a global solution, then the life span of solutions with approximating initial data tends to $+\infty$ (see Corollary~\ref{cor:AGE}).

This improves the classical result (see for example~\cite[Theorem~2.1]{1996-TAMS-AroPan}) according to which the life span of solutions is locally bounded from below in $D(A^{3/4})\times D(A^{1/4})$, and solutions depend continuously on initial data in some common time interval (but not necessarily in the longest possible time interval).

\item  From the quantitative point of view, in Theorem~\ref{thm:AGE} we provide estimates from below for $T_{\ep}$ when $(u_{0},u_{1})$ belongs to three special classes of initial data for which global existence is known, namely data that are finite linear combinations of eigenvectors of $A$, analytic data, and quasi-analytic data.

\end{itemize}

The quantitative estimates are rather weak, for example in the classical quasi-analytic case they involve three nested logarithms, but we point out that in this result we have no smallness assumptions on $(u_{0\ep},u_{1\ep})$. 

From the technical point of view, our estimates from below for $T_{\ep}$ are based on some estimates from above for the growth in time of the energy of the solution with the limiting initial condition $(u_{0},u_{1})$ (see Theorem~\ref{thm:growth}). These growth estimates are a by-product of the classical global existence results in analytic and quasi-analytic classes, but probably they had not been stated or proved explicitly before. 

Our technique, when applied to the null solution, provides the basic estimate (\ref{th:ep2}) (see Remark~\ref{rmk:basic}), and therefore we can not exclude that more refined techniques can lead to longer life spans also in this more general framework. The race for better estimates in on.

\paragraph{\textmd{\textit{Structure of the paper}}}

This paper is organized as follows. In section~\ref{sec:statements} we fix the functional setting and we state our main results concerning the lower semicontinuity of the life span, the energy estimates for global solutions, and the qualitative and quantitative almost global existence. In section~\ref{sec:quantitative} we concentrate the quantitative estimates that represent the technical core of the paper. In section~\ref{sec:proofs} we prove the main results. 

\setcounter{equation}{0}
\section{Statements}\label{sec:statements}

Let us start by clarifying the functional setting. In this paper we assume that
\begin{list}
{}{\setlength{\leftmargin}{4em}\setlength{\labelwidth}{4em}}

\item[(Hp-$H$)] $H$ is a real Hilbert space;

\item[(Hp-$A$)] $A$ is a self-adjoint operator on $H$ with domain $D(A)$, and there exist an orthonormal basis $\{e_{k}\}_{k\geq 1}$ of $H$, and a sequence $\{\lambda_{k}\}$ of nonnegative real numbers such that 
\begin{equation}
Ae_{k}=\lambda_{k}^{2}e_{k}
\qquad
\forall k\geq 1;
\nonumber
\end{equation}

\item[(Hp-$m$)] $m:[0,+\infty)\to\re$ is a function of class $C^{1}$ that satisfies the strict hyperbolicity assumption (\ref{hp:m-sh}).

\end{list}

The theory can be extended to more general cases that require only rather standard technical adjustments without introducing new ideas. For example, concerning the operator $A$, it is enough to assume that it is unitary equivalent to a multiplication operator on some $L^{2}$ space (see~\cite[Remark~2.7]{gg:EnergySpaceSG} for more details); concerning the nonlinearity $m$, one can replace the $C^{1}$ regularity by Lipschitz regularity.

Let us recall the classical local existence result (see~\cite[Theorem~2.1]{1996-TAMS-AroPan}).

\begin{thmbibl}[Local-in-time existence]\label{thmbibl:loc-ex}

Let us assume that $H$, $A$, and $m$ satisfy the hypotheses stated at the beginning of section~\ref{sec:statements}.

Then the following statements hold true.
\begin{enumerate}
\renewcommand{\labelenumi}{(\arabic{enumi})}

\item  \emph{(Local-in-time existence)}. For every $(u_{0},u_{1})\in D(A^{3/4})\times D(A^{1/4})$ there exists a real number $T>0$ such that  problem (\ref{K:eqn})--(\ref{K:data}) has a solution in the class (\ref{defn:strong-sol}).

\item  \emph{(Uniqueness)}. The solution is unique among solutions with the regularity  (\ref{defn:strong-sol}).

\item  \emph{(Energy equality)}. If we set
\begin{equation}
M(\sigma):=\int_{0}^{\sigma}m(s)\,ds
\qquad
\forall\sigma\geq 0,
\label{defn:M}
\end{equation}
then the classical Hamiltonian defined by
\begin{equation}
H(t):=|u'(t)|^{2}+M\left(|A^{1/2}u(t)|^{2}\right)
\label{defn:ham}
\end{equation}
is constant along trajectories, namely it satisfies the energy equality
\begin{equation}
|u'(t)|^{2}+M\left(|A^{1/2}u(t)|^{2}\right)=|u_{1}|^{2}+M\left(|A^{1/2}u_{0}|^{2}\right)
\qquad
\forall t\in[0,T].
\nonumber
\end{equation}

\item  \emph{(Propagation of regularity)}. If in addition $(u_{0},u_{1})\in D(A^{\alpha+3/4})\times D(A^{\alpha+1/4})$ for some real number $\alpha\geq 0$, then the solution belongs to
\begin{equation}
C^{0}\left([0,T],D(A^{\alpha+3/4})\right)\cap C^{1}\left([0,T],D(A^{\alpha+1/4})\strut\right).
\nonumber
\end{equation}

\item  \emph{(Continuation and alternative)}. The solution can be continued to a solution defined in some maximal interval $[0,T_{\mathrm{max}})$, where either $T_{\mathrm{max}}=+\infty$ or
\begin{equation}
\limsup_{t\to T_{\mathrm{max}}^{-}}\left(|A^{1/4}u'(t)|^{2}+|A^{3/4}u(t)|^{2}\right)=+\infty.
\label{th:alternative}
\end{equation} 

\end{enumerate}

\end{thmbibl}

The last statement motivates the following notion.

\begin{defn}[Life span]
\begin{em}

For every $(u_{0},u_{1})\in D(A^{3/4})\times D(A^{1/4})$ we define $\LS(u_{0},u_{1})$ as the \emph{life span} of the solution to problem (\ref{K:eqn})--(\ref{K:data}), namely the value $T_{\mathrm{max}}$ that appears in statement~(5) of Theorem~\ref{thmbibl:loc-ex}. The life span is either a positive real number or $+\infty$.

\end{em}
\end{defn}

For every pair $(z_{0},z_{1})\in D(A^{3/4})\times D(A^{1/4})$ we define
\begin{equation}
E(z_{0},z_{1}):=|z_{1}|^{2}+|A^{1/4}z_{1}|^{2}+|A^{1/2}z_{0}|^{2}+|A^{3/4}z_{0}|^{2},
\nonumber
\end{equation}
and for every function $z$ in the space (\ref{defn:strong-sol}) we set
\begin{equation}
E_{z}(t):=E(z(t),z'(t))=|z'(t)|^{2}+|A^{1/4}z'(t)|^{2}+|A^{1/2}z(t)|^{2}+|A^{3/4}z(t)|^{2}.
\label{defn:E(t)}
\end{equation}

We are now ready to state the first result of this paper.


\begin{thm}[Lower semicontinuity of the life span and continuous dependence on initial data on the whole life span]\label{thm:lsc}

Let us assume that $H$, $A$, and $m$ satisfy the hypotheses stated at the beginning of section~\ref{sec:statements}.

Let $(u_{0\ep},u_{1\ep})\to (u_{0},u_{1})$ in $D(A^{3/4})\times D(A^{1/4})$ be a converging family of initial data, let $u_{\ep}(t)$ and $u(t)$ denote the corresponding solutions to (\ref{K:eqn}), and let $\LS(u_{0\ep},u_{1\ep})$ and $\LS(u_{0},u_{1})$ denote the life span of these solutions (possibly equal to $+\infty$).

Then the following statements hold true.
\begin{enumerate}
\renewcommand{\labelenumi}{(\arabic{enumi})}

\item  \emph{(Lower semicontinuity of the life span).} It turns out that 
\begin{equation}
\liminf_{\ep\to 0^{+}}\LS(u_{0\ep},u_{1\ep})\geq\LS(u_{0},u_{1}).
\label{th:LS-lsc}
\end{equation}

\item  \emph{(Continuous dependence on the initial condition).} For every $T<\LS(u_{0},u_{1})$ it turns out that $u_{\ep}(t)\to u(t)$ in the sense that
\begin{equation}
E_{(u_{\ep}-u)}(t)\to 0
\qquad
\text{uniformly in }[0,T],
\label{th:u-uep-noreg}
\end{equation}
where $E_{(u_{\ep}-u)}(t)$ is defined according to (\ref{defn:E(t)}). 

\item  \emph{(``Lipschitz like'' convergence rate when $(u_{0},u_{1})$ is more regular).} Let us assume in addition that $(u_{0},u_{1})\in D(A^{5/4})\times D(A^{3/4})$. 

Then for every $T<\LS(u_{0},u_{1})$ there exists a constant $\Gamma$ such that
\begin{equation}
E_{(u_{\ep}-u)}(t)\leq\Gamma\cdot E_{(u_{\ep}-u)}(0)
\qquad
\forall t\in[0,T].
\label{th:u-uep-reg}
\end{equation}

\end{enumerate}

\end{thm}

Theorem~\ref{thm:lsc} above implies the following almost global existence result, which we call qualitative because it comes with no estimate on the life spans, but also no special assumption on the limiting initial condition (we just assume that for some reason it originates a global solution).

\begin{cor}[Qualitative almost global existence]\label{cor:AGE}

Let us assume that $H$, $A$, and $m$ satisfy the hypotheses stated at the beginning of section~\ref{sec:statements}.

Let us assume that, for some initial condition $(u_{0},u_{1})\in D(A^{3/4})\times D(A^{1/4})$, problem (\ref{K:eqn})--(\ref{K:data}) admits a global solution.

Then for every family $\{(u_{0\ep},u_{1\ep})\}\subseteq D(A^{3/4})\times D(A^{1/4})$ of initial data that converges to $(u_{0},u_{1})$ in the same space it turns out that
\begin{equation}
\lim_{\ep\to 0^{+}}\LS(u_{0\ep},u_{1\ep})=+\infty,
\nonumber
\end{equation}
and $E_{(u_{\ep}-u)}(t)\to 0$ uniformly on all bounded time intervals.
\end{cor}

Now we recall the notion of generalized Sobolev-Gevrey spaces with respect to an operator.

\begin{defn}[Generalized Sobolev-Gevrey spaces]
\begin{em}

Let us assume that $H$ and $A$ satisfy the hypotheses stated at the beginning of section~\ref{sec:statements}. Let $\alpha$ be a nonnegative real number, and let $\varphi:[0,+\infty)\to\re$ be a function.

The space $\mathcal{G}_{\varphi,\alpha}(A)$ is the set of all vectors $z\in H$ such that
\begin{equation}
\sum_{k=1}^{\infty}\lambda_{k}^{4\alpha}\exp(\varphi(\lambda_{k}))\langle z,e_{k}\rangle^{2}<+\infty.
\label{defn:SG-spaces}
\end{equation}

\end{em}
\end{defn}

We observe that when $\varphi(\sigma)\equiv 0$ we obtain the usual ``Sobolev spaces'' $D(A^{\alpha})$, while in the case where $\alpha=0$ and $\varphi(\sigma)=r_{0}\sigma^{1/s}$ for some $r_{0}>0$ we obtain the usual ``Gevrey spaces'' of order $s$ with ``radius'' $r_{0}$. We also observe that the convergence of the series in (\ref{defn:SG-spaces}) is not affected by adding a constant to $\varphi$, and therefore in the sequel we always assume, without loss of generality, that $\varphi(0)=0$.

We are now ready to recall the classical result concerning global existence for quasi-analytic initial data. We state both the existence result (see~\cite[Theorem~2.1]{gg:K-Nishihara}), and the key estimate in the proof (see~\cite[section~3.2.3.]{gg:K-Nishihara}), because we need it in the sequel.

\begin{thmbibl}[Global existence in quasi-analytic classes]\label{thmbibl:glob-ex}

Let us assume that $H$, $A$, and $m$ satisfy the hypotheses stated at the beginning of section~\ref{sec:statements}. Let $\varphi:[0,+\infty)\to[0,+\infty)$ be an increasing continuous function such that $\varphi(0)=0$ and
\begin{equation}
\int_{1}^{+\infty}\frac{\varphi(\sigma)}{\sigma^{2}}\,d\sigma=+\infty.
\nonumber
\end{equation}

Then the following statements hold true.
\begin{enumerate}
\renewcommand{\labelenumi}{(\arabic{enumi})}

\item  \emph{(Global existence).} For every initial condition 
\begin{equation}
(u_{0},u_{1})\in\mathcal{G}_{\varphi,3/4}(A)\times\mathcal{G}_{\varphi,1/4}(A),
\label{hp:data-phi}
\end{equation}
problem (\ref{K:eqn})--(\ref{K:data}) admits a unique global solution, and this solution belongs to the space
\begin{equation}
C^{0}\left([0,+\infty),\mathcal{G}_{\varphi,3/4}(A)\right)\cap
C^{1}\left([0,+\infty),\mathcal{G}_{\varphi,1/4}(A)\right).
\nonumber
\end{equation}

\item  \emph{(Differential inequality for the energy).} Let us consider the corrected $\varphi$-energy defined by
\begin{equation}
F_{\varphi}(t):=\sum_{k=1}^{\infty}\max\{1,\lambda_{k}\}
\cdot a_{k}(t)\cdot
\exp(\varphi(\lambda_{k}))
\label{defn:F-phi}
\end{equation}
where
\begin{equation}
a_{k}(t):=\langle u'(t),e_{k}\rangle^{2}+
m\left(|A^{1/2}u(t)|^{2}\right)\cdot\lambda_{k}^{2}\langle u(t),e_{k}\rangle^{2}.
\nonumber
\end{equation}

If $|u_{1}|^{2}+|A^{1/2}u_{0}|^{2}>0$, then there exist real number $c_{0}$ and $c_{1}$ such that this energy satisfies the estimate
\begin{equation}
F_{\varphi}(t)\geq c_{1}>0
\qquad
\forall t\geq 0,
\nonumber
\end{equation}
and the differential inequality
\begin{equation}
F_{\varphi}'(t)\leq c_{0}F_{\varphi}(t)\cdot
\left\{1+\varphi^{-1}\left(
\log\frac{F_{\varphi}(t)}{c_{1}}\right)\right\}
\qquad
\forall t\geq 0.
\label{th:diff-ineq-F}
\end{equation}

\end{enumerate}

\end{thmbibl}




The second result of this paper concerns energy estimates for global solutions to (\ref{K:eqn}) provided by Theorem~\ref{thmbibl:glob-ex} above, at least for the two main examples of weights $\varphi$ (but the method can be extended to more general choices).

\begin{thm}[Energy estimates for (quasi) analytic global solutions]\label{thm:growth}

Let us assume that $H$, $A$, and $m$ satisfy the hypotheses stated at the beginning of section~\ref{sec:statements}.

Then the following statements hold true.
\begin{enumerate}
\renewcommand{\labelenumi}{(\arabic{enumi})}

\item  \emph{(Analytic data).} Let $r_{0}$ be a positive real number. Let us consider the global solution to problem (\ref{K:eqn})--(\ref{K:data}) with initial data that satisfy (\ref{hp:data-phi}) with $\varphi(\sigma)=r_{0}\sigma$ for every $\sigma\geq 0$. 

Then there exists a positive real number $\beta_{1}$ such that the energy defined by (\ref{defn:F-phi}) satisfies the estimate
\begin{equation}
F_{\varphi}(t)\leq F_{\varphi}(0)\exp(\exp(\beta_{1}t))
\qquad
\forall t\geq 0.
\label{th:F-an}
\end{equation}

Moreover, for every real number $\alpha\geq 0$ there exist a positive real number $B_{1,\alpha}$ such that
\begin{equation}
|A^{\alpha}u'(t)|^{2}+|A^{\alpha+1/2}u(t)|^{2}\leq 
B_{1,\alpha}\exp(4\alpha\beta_{1}t)
\qquad
\forall t\geq 0.
\label{th:an-est}
\end{equation}

\item  \emph{(Classical quasi analytic data).} Let us consider the global solution to problem (\ref{K:eqn})--(\ref{K:data}) with initial data that satisfy (\ref{hp:data-phi}) with 
\begin{equation}
\varphi(\sigma):=\frac{\sigma}{\log(2+\sigma)}
\qquad
\forall\sigma\geq 0.
\label{defn:qa}
\end{equation}

Then there exists a positive real number $\beta_{2}$ such that the energy defined by (\ref{defn:F-phi}) satisfies the estimate
\begin{equation}
F_{\varphi}(t)\leq F_{\varphi}(0)\exp(\exp(\exp(\beta_{2}t)))
\qquad
\forall t\geq 0.
\label{th:F-qa}
\end{equation}

Moreover, for every real number $\alpha\geq 0$ there exist a positive real number $B_{2,\alpha}$ such that
\begin{equation}
|A^{\alpha}u'(t)|^{2}+|A^{\alpha+1/2}u(t)|^{2}\leq 
B_{2,\alpha}\exp\left(\exp\left(\max\{2,4\alpha\}\beta_{2}t\strut\right)\right)
\qquad
\forall t\geq 0.
\label{th:qa-est}
\end{equation}

\end{enumerate}

\end{thm}


The third and last result of this paper is an estimate from below for the life span of solutions with initial data that are close to some special classes of data that originate global solutions.

\begin{thm}[Quantitative almost global existence in special cases]\label{thm:AGE}

Let us assume that $H$, $A$, and $m$ satisfy the hypotheses stated at the beginning of section~\ref{sec:statements}.

Let $(u_{0},u_{1})\in D(A^{3/4})\times D(A^{1/4})$, and for every $\ep>0$ let $(u_{0\ep},u_{1\ep})\in D(A^{3/4})\times D(A^{1/4})$ be such that
\begin{equation}
|u_{1}-u_{1\ep}|^{2}+
|A^{1/4}(u_{1}-u_{1\ep})|^{2}+
|A^{1/2}(u_{0}-u_{0\ep})|^{2}+
|A^{3/4}(u_{0}-u_{0\ep})|^{2}\leq
\ep^{2}.
\nonumber
\end{equation}

Let $\LS(u_{0},u_{1})$ and $\LS(u_{0\ep},u_{1\ep})$ denote the life span (possibly equal to $+\infty$) of the corresponding solutions to (\ref{K:eqn}).

Then the following estimates hold true.
\begin{enumerate}
\renewcommand{\labelenumi}{(\arabic{enumi})}

\item  \emph{(Finite dimensional initial data).} Let us assume that $u_{0}$ and $u_{1}$ are finite linear combinations of eigenvalues of $A$.

Then there exist positive real numbers $\gamma_{0}$ and $\ep_{0}$ such that
\begin{equation}
\LS(u_{0\ep},u_{1\ep})\geq\frac{1}{\gamma_{0}}\log|\ep|
\qquad
\forall\ep\in(0,\ep_{0}).
\nonumber
\end{equation}

\item  \emph{(Analytic initial data).} Let $r_{0}$ be a positive real number. Let us assume that $(u_{0},u_{1})$ satisfy (\ref{hp:data-phi}) with $\varphi(\sigma):=r_{0}\sigma$ for all $\sigma\geq 0$.

Then there exist positive real numbers $\gamma_{1}$ and $\ep_{1}$ such that
\begin{equation}
\LS(u_{0\ep},u_{1\ep})\geq\frac{1}{\gamma_{1}}\log(|\log\ep|)
\qquad
\forall\ep\in(0,\ep_{1}).
\nonumber
\end{equation}

\item  \emph{(Classical quasi-analytic initial data).} Let us assume that $(u_{0},u_{1})$ satisfy (\ref{hp:data-phi}) with $\varphi$ defined by (\ref{defn:qa}).

Then there exist positive real numbers $\gamma_{2}$ and $\ep_{2}$ such that
\begin{equation}
\LS(u_{0\ep},u_{1\ep})\geq\frac{1}{\gamma_{2}}\log(\log(|\log\ep|))
\qquad
\forall\ep\in(0,\ep_{2}).
\nonumber
\end{equation}

\end{enumerate}

\end{thm}


\setcounter{equation}{0}
\section{Quantitative estimates}\label{sec:quantitative}

This section is the technical core of this paper. The first result is the classical energy estimate for solutions of an abstract \emph{linear} wave equation with a time-dependent propagation speed.

\begin{lemma}\label{lemma:basic}

Let us assume that $H$, $A$, and $m$ satisfy the hypotheses stated at the beginning of section~\ref{sec:statements}.

Let $S_{0}$ be a positive real number, and let $c:[0,S_{0}]\to\re$ be a function of class $C^{1}$. Let us assume that there exists three real numbers $\nu_{0}$, $C_{0}$, $\Lambda_{0}$ such that
\begin{equation}
0<\nu_{0}\leq c(t)\leq C_{0}
\qquad
\forall t\in[0,S_{0}],
\label{hp:bound-c}
\end{equation}
and
\begin{equation}
|c'(t)|\leq\Lambda_{0}
\qquad
\forall t\in[0,S_{0}].
\nonumber
\end{equation}

Let $\alpha\geq 0$ be a real number, and let
\begin{equation}
z\in C^{1}\left([0,S_{0}],D(A^{\alpha})\right)\cap C^{0}\left([0,S_{0}],D(A^{\alpha+1/2})\right)
\nonumber
\end{equation}
be a function such that
\begin{equation}
z''(t)+c(t)Az(t)=0
\qquad
\forall t\in[0,S_{0}].
\nonumber
\end{equation}

Then for every $t\in[0,S_{0}]$ it turns out that
\begin{equation}
|A^{\alpha}z'(t)|^{2}+|A^{\alpha+1/2}z(t)|^{2}\leq
\left(|A^{\alpha}z'(0)|^{2}+|A^{\alpha+1/2}z(0)|^{2}\right)
\frac{\max\{1,C_{0}\}}{\min\{1,\nu_{0}\}}
\exp\left(\frac{\Lambda_{0}}{\nu_{0}}t\right).
\label{th:lemma-0}
\end{equation} 

\end{lemma}

\begin{proof}

Let us consider the $\alpha$-energy
\begin{equation}
\mathcal{E}_{\alpha}(t):=|A^{\alpha}z'(t)|^{2}+|A^{\alpha+1/2}z(t)|^{2},
\nonumber
\end{equation}
and the modified $\alpha$-energy
\begin{equation}
\widehat{\mathcal{E}}_{\alpha}(t):=|A^{\alpha}z'(t)|^{2}+c(t)|A^{\alpha+1/2}z(t)|^{2}.
\nonumber
\end{equation}

From the bounds in (\ref{hp:bound-c}) it follows that the two energies are equivalent in the sense that
\begin{equation}
\min\{1,\nu_{0}\}\mathcal{E}_{\alpha}(t)\leq
\widehat{\mathcal{E}}_{\alpha}(t)\leq
\max\{1,C_{0}\}\mathcal{E}_{\alpha}(t)
\qquad
\forall t\in[0,S_{0}].
\label{est:G-equiv}
\end{equation}

The time-derivative of the modified energy is equal to
\begin{equation}
\widehat{\mathcal{E}}_{\alpha}'(t)=
c'(t)|A^{\alpha+1/2}z(t)|^{2}=
\frac{c'(t)}{c(t)}\cdot c(t)|A^{\alpha+1/2}z(t)|^{2}\leq
\frac{\Lambda_{0}}{\nu_{0}}\cdot\widehat{\mathcal{E}}_{\alpha}(t).
\nonumber
\end{equation}

Integrating this differential inequality we conclude that
\begin{equation}
\widehat{\mathcal{E}}_{\alpha}(t)\leq\widehat{\mathcal{E}}_{\alpha}(0)\exp\left(\frac{\Lambda_{0}}{\nu_{0}}t\right)
\qquad
\forall t\in[0,S_{0}],
\nonumber
\end{equation}
which implies (\ref{th:lemma-0}) because of (\ref{est:G-equiv}).
\end{proof}


The second result concerns continuous dependence on initial data. We consider two solutions to equation (\ref{K:eqn}) defined in the same time interval $[0,T]$, and we estimate their difference at time $t$ in terms of their difference at time~0. This estimate depends on the growth of the two solutions in the interval $[0,T]$, and we need to estimate explicitly this dependence.

The result turns out to be simpler, and the convergence rate to be better, if at least one of the two solutions is slightly more regular. This is a very well-know issue also in the case where one considers the difference between solutions to two different linear wave equations with different propagation speeds.

In order to deal with the case with minimal regularity, we need to introduce some notation. For every real number $\lambda>0$, and every vector $z\in H$, we define the low-frequency component $z_{\lambda,-}$ of $z$, and the high-frequency component $z_{\lambda,+}$ of $z$, as
\begin{equation}
z_{\lambda,-}:=\sum_{\lambda_{k}\leq\lambda}\langle z,e_{k}\rangle e_{k},
\qquad\qquad
z_{\lambda,+}:=\sum_{\lambda_{k}>\lambda}\langle z,e_{k}\rangle e_{k}.
\nonumber
\end{equation}

With this notation we can split (\ref{defn:E(t)}) as
\begin{equation}
E_{z}(t)=E_{z}^{\lambda,-}(t)+E_{z}^{\lambda,+}(t),
\nonumber
\end{equation}
where 
\begin{equation}
E_{z}^{\lambda,-}(t):=E_{(z_{\lambda,-})}(t)
\qquad\qquad\text{and}\qquad\qquad
E_{z}^{\lambda,+}(t):=E_{(z_{\lambda,+})}(t)
\nonumber
\end{equation}
denote the contribution to $E_{z}(t)$ of the low-frequency and high-frequency components of $z(t)$.

We are now ready to state the result.


\begin{prop}[Quantitative well-posedness]\label{prop:quant-wp}

Let us assume that $H$, $A$, and $m$ satisfy the hypotheses stated at the beginning of section~\ref{sec:statements}.

Let $S_{0}$ be a positive real number, and let $u$ and $v$ be two solutions to equation (\ref{K:eqn}) in the space 
\begin{equation}
C^{1}\left([0,S_{0}],D(A^{1/4})\right)\cap C^{0}\left([0,S_{0}],D(A^{3/4})\right).
\label{hp:reg-u}
\end{equation}

Let $R_{0}$ be a real number such that
\begin{equation}
\max\left\{|A^{1/2}u(t)|,|A^{1/2}v(t)|\right\}\leq R_{0}
\qquad
\forall t\in[0,S_{0}],
\label{hp:R0}
\end{equation}
and let us define the constants
\begin{equation}
C_{0}:=\max\left\{m(\sigma):0\leq\sigma\leq R_{0}^{2}\right\},
\qquad
L_{0}:=\max\left\{|m'(\sigma)|:0\leq\sigma\leq R_{0}^{2}\right\}.
\label{defn:C0-L0}
\end{equation}

Let $R_{1}$ be a real number such that
\begin{equation}
\max\left\{|A^{1/4}u'(t)|,|A^{3/4}u(t)|\right\}\leq R_{1}\qquad
\forall t\in[0,S_{0}]
\label{hp:R1-u}
\end{equation}
and
\begin{equation}
\max\left\{|A^{1/4}v'(t)|,|A^{3/4}v(t)|\right\}\leq 2R_{1}
\qquad
\forall t\in[0,S_{0}].
\label{hp:R1-v}
\end{equation}

Then the following statements hold true.
\begin{enumerate}
\renewcommand{\labelenumi}{(\arabic{enumi})}

\item  \emph{(More regular case).} Let us assume in addition that $u\in C^{0}\left([0,S_{0}],D(A^{5/4})\right)$, and let $R_{2}$ be a real number such that
\begin{equation}
|A^{5/4}u(t)|\leq R_{2}
\qquad
\forall t\in[0,S_{0}].
\label{hp:R2-reg}
\end{equation}

Then it turns out that
\begin{equation}
E_{(u-v)}(t)\leq
\Gamma_{1}E_{(u-v)}(0)\exp(\Gamma_{2}t)
\qquad
\forall t\in[0,S_{0}],
\label{th:wp-reg}
\end{equation}
where the constants $\Gamma_{1}$ and $\Gamma_{2}$ are defined by
\begin{equation}
\Gamma_{1}:=\frac{\max\{1,C_{0}\}}{\min\{1,\nu_{0}\}},
\qquad\qquad
\Gamma_{2}:=\frac{8L_{0}R_{1}^{2}}{\nu_{0}}+\frac{4L_{0}R_{0}(R_{1}+R_{2})}{\sqrt{\nu_{0}}}.
\label{defn:G12-reg}
\end{equation}

\item  \emph{(Minimal regularity case).} Let us assume that $u$ has just the regularity (\ref{hp:reg-u}). Let $\lambda>0$ be any positive real number, and let $R_{2,\lambda}$ be a real number such that (note that now the left-hand side involves only the low-frequency components of $u$, and that the constant in the right-hand side depends on $\lambda$)
\begin{equation}
|A^{5/4}u_{\lambda,-}(t)|\leq R_{2,\lambda}
\qquad
\forall t\in[0,S_{0}].
\label{hp:R2-noreg}
\end{equation}

Then for every $t\in[0,S_{0}]$ it turns out that
\begin{equation}
E_{(u-v)}(t)\leq
\Gamma_{1,\lambda}E_{(u-v)}(0)\exp(\Gamma_{2,\lambda}t)+
\Gamma_{3}\left(E_{u}^{\lambda,+}(0)+E_{v}^{\lambda,+}(0)\right)\exp(\Gamma_{4}t),
\label{th:wp-noreg}
\end{equation}
where the ($\lambda$-dependent) constants $\Gamma_{1,\lambda}$ and $\Gamma_{2,\lambda}$ are defined by
\begin{equation}
\Gamma_{1,\lambda}:=\frac{\max\{1,C_{0}\}}{\min\{1,\nu_{0}\}}\max\left\{1,\frac{1}{\lambda^{2}}\right\}
\label{defn:G1-noreg}
\end{equation}
and
\begin{equation}
\Gamma_{2,\lambda}:=
\frac{8L_{0}R_{1}^{2}}{\nu_{0}}+\frac{2L_{0}(2R_{0}+3R_{1})(2R_{1}+R_{2,\lambda})}{\sqrt{\nu_{0}}},
\label{defn:G2-noreg}
\end{equation}
while the ($\lambda$-independent) constants $\Gamma_{3}$ and $\Gamma_{4}$ are defined by
\begin{equation}
\Gamma_{3}:=2\frac{\max\{1,C_{0}\}}{\min\{1,\nu_{0}\}},
\qquad\qquad
\Gamma_{4}:=\frac{8L_{0}R_{1}^{2}}{\nu_{0}}.
\label{defn:G34}
\end{equation}

\end{enumerate}

\end{prop}

\begin{proof}

For every $t\in[0,S_{0}]$ we set for simplicity
\begin{equation}
w(t):=u(t)-v(t),
\qquad
c_{u}(t):=m\left(|A^{1/2}u(t)|^{2}\right),
\qquad
c_{v}(t):=m\left(|A^{1/2}v(t)|^{2}\right).
\nonumber
\end{equation}

Since $u$ and $v$ are solutions to equation (\ref{K:eqn}), with these notations it turns out that
\begin{equation}
w''(t)+c_{v}(t)Aw(t)=(c_{v}(t)-c_{u}(t))Au(t)
\qquad
\forall t\in[0,S_{0}].
\label{eqn:w}
\end{equation}

Due to (\ref{hp:m-sh}), (\ref{hp:R0}) and (\ref{defn:C0-L0}), the function $c_{v}$ satisfies the estimates
\begin{equation}
0<\nu_{0}\leq c_{v}(t)\leq C_{0}
\qquad
\forall t\in[0,S_{0}],
\label{est:bound-cv}
\end{equation}
and the same is true for $c_{u}(t)$. Moreover, we observe that
\begin{equation}
c_{v}'(t)=m'\left(|A^{1/2}v(t)|^{2}\right)\cdot 2\langle A^{1/4}v'(t),A^{3/4}v(t)\rangle,
\nonumber
\end{equation}
and therefore from (\ref{defn:C0-L0}) we deduce that
\begin{equation}
|c_{v}'(t)|\leq 2L_{0}|A^{1/4}v'(t)|\cdot|A^{3/4}v(t)|
\qquad
\forall t\in[0,S_{0}],
\label{est:cv'-basic}
\end{equation}
and an analogous estimate is true for $c_{u}'(t)$. Finally, the Lipschitz continuity of $m$ in the interval $[0,R_{0}^{2}]$ implies in particular that
\begin{equation}
|c_{v}(t)-c_{u}(t)|\leq L_{0}\left||A^{1/2}u(t)|^{2}-|A^{1/2}v(t)|^{2}\right|
\qquad
\forall t\in[0,S_{0}].
\label{est:cu-cv}
\end{equation}

We are now ready to prove the two statements. 

\subparagraph{\textmd{\textit{Proof of statement~(1)}}}

Let us consider the modified energy
\begin{equation}
F(t):=|w'(t)|^{2}+|A^{1/4}w'(t)|^{2}+
c_{v}(t)\left(|A^{1/2}w(t)|^{2}+|A^{3/4}w(t)|^{2}\right).
\nonumber
\end{equation}

Due to (\ref{est:bound-cv}), this energy is equivalent to $E_{w}(t)$ defined by (\ref{defn:E(t)}) in the sense that
\begin{equation}
\min\{1,\nu_{0}\}E_{w}(t)\leq F(t)\leq\max\{1,C_{0}\}E_{w}(t)
\qquad
\forall  t\in[0,S_{0}].
\label{est:EF-equiv}
\end{equation}

From (\ref{eqn:w}) it follows that the time-derivative of $F(t)$ is equal to
\begin{eqnarray*}
F'(t) & = &
c_{v}'(t)\left(|A^{1/2}w(t)|^{2}+|A^{3/4}w(t)|^{2}\right)
\\
& &
\mbox{}+2(c_{v}(t)-c_{u}(t))\left(\langle w'(t),Au(t)\rangle+\langle A^{1/4}w'(t),A^{5/4}u(t)\rangle\right).
\end{eqnarray*}

Let us estimate the terms in the right-hand side. From (\ref{est:cv'-basic}) and (\ref{hp:R1-v}) we know that
\begin{equation}
|c_{v}'(t)|\leq 8L_{0}R_{1}^{2},
\label{est:cv'}
\end{equation}
and therefore from the bound from below in (\ref{est:bound-cv}) we obtain that
\begin{eqnarray*}
c_{v}'(t)\left(|A^{1/2}w(t)|^{2}+|A^{3/4}w(t)|^{2}\right) & = & 
\frac{c_{v}'(t)}{c_{v}(t)}\cdot c_{v}(t)\left(|A^{1/2}w(t)|^{2}+|A^{3/4}w(t)|^{2}\right)
\\
& \leq &
\frac{8L_{0}R_{1}^{2}}{\nu_{0}}F(t).
\end{eqnarray*}

Similarly, from (\ref{hp:R0}) we obtain that
\begin{eqnarray*}
\left||A^{1/2}u(t)|^{2}-|A^{1/2}v(t)|^{2}\right| & = & 
\left|\langle A^{1/2}w(t),A^{1/2}(u(t)+v(t))\rangle\right|
\\
& \leq &
\frac{1}{c_{v}(t)^{1/2}}\cdot c_{v}(t)^{1/2}|A^{1/2}w(t)|\cdot\left(|A^{1/2}u(t)|+|A^{1/2}v(t)|\right)
\\
& \leq & 
\frac{1}{\sqrt{\nu_{0}}}\cdot F(t)^{1/2}\cdot 2R_{0},
\end{eqnarray*}
and therefore from (\ref{est:cu-cv}) we deduce that
\begin{equation}
|c_{u}(t)-c_{v}(t)|\leq\frac{2L_{0}R_{0}}{\sqrt{\nu_{0}}}F(t)^{1/2}.
\nonumber
\end{equation}

Finally, from (\ref{hp:R1-u}) and (\ref{hp:R2-reg}) we obtain that
\begin{multline*}
\quad\left|\langle w'(t),Au(t)\rangle+\langle A^{1/4}w'(t),A^{5/4}u(t)\rangle\right| \leq
\\[0.5ex]
|A^{1/4}w'(t)|\cdot|A^{3/4}u(t)|+|A^{1/4}w'(t)|\cdot|A^{5/4}u(t)|
\leq 
(R_{1}+R_{2})F(t)^{1/2}.\quad
\end{multline*}

From all these estimates we deduce that
\begin{equation}
F'(t)\leq\left(\frac{8L_{0}R_{1}^{2}}{\nu_{0}}+\frac{4L_{0}R_{0}(R_{1}+R_{2})}{\sqrt{\nu_{0}}}\right)F(t).
\qquad
\forall  t\in[0,S_{0}],
\nonumber
\end{equation}
and integrating this differential inequality we conclude that
\begin{equation}
F(t)\leq F(0)\exp\left(\left(\frac{8L_{0}R_{1}^{2}}{\nu_{0}}+\frac{4L_{0}R_{0}(R_{1}+R_{2})}{\sqrt{\nu_{0}}}\right)t\right)
\qquad
\forall  t\in[0,S_{0}],
\nonumber
\end{equation}
which implies (\ref{th:wp-reg}) because of the equivalence (\ref{est:EF-equiv}).

\subparagraph{\textmd{\textit{Proof of statement~(2)}}}

Let us consider the modified energy
\begin{eqnarray*}
F_{\lambda}(t) & := & 
|w_{\lambda,-}'(t)|^{2}+|A^{1/4}w_{\lambda,-}'(t)|^{2}+|A^{-1/4}w_{\lambda,+}'(t)|^{2}
\\
& & 
\mbox{}+c_{v}(t)\left(|A^{1/2}w_{\lambda,-}(t)|^{2}+|A^{3/4}w_{\lambda,-}(t)|^{2}+|A^{1/4}w_{\lambda,+}(t)|^{2}\right).
\end{eqnarray*}

From the estimate from below in (\ref{est:bound-cv}) it follows that
\begin{equation}
F_{\lambda}(t)\geq\min\{1,\nu_{0}\}E_{w}^{\lambda,-}(t)
\qquad
\forall t\in[0,S_{0}].
\label{equiv:F>E}
\end{equation}

Moreover, since
\begin{equation}
|A^{-1/4}w_{\lambda,+}'(t)|^{2}+|A^{1/4}w_{\lambda,+}(t)|^{2}\leq
\frac{1}{\lambda^{2}}\left(|A^{1/4}w_{\lambda,+}'(t)|^{2}+|A^{3/4}w_{\lambda,+}(t)|^{2}\right),
\nonumber
\end{equation}
from the estimate from above in (\ref{est:bound-cv}) it follows that
\begin{equation}
F_{\lambda}(t)\leq\max\{1,C_{0}\}\max\left\{1,\frac{1}{\lambda^{2}}\right\}E_{w}(t)
\qquad
\forall t\in[0,S_{0}].
\label{equiv:F<E}
\end{equation}

The time-derivative of $F_{\lambda}(t)$ is equal to
\begin{eqnarray*}
F_{\lambda}'(t) & = &
c_{v}'(t)\left(|A^{1/2}w_{\lambda,-}(t)|^{2}+|A^{3/4}w_{\lambda,-}(t)|^{2}+|A^{1/4}w_{\lambda,+}(t)|^{2}\right)
\\[0.5ex]
& & 
\mbox{}+2(c_{v}(t)-c_{u}(t))\cdot\left(
\langle w_{\lambda,-}'(t),Au_{\lambda,-}(t)\rangle\right.
\\[0.5ex]
& & 
\mbox{}+\left.
\langle A^{1/4}w_{\lambda,-}'(t),A^{5/4}u_{\lambda,-}(t)\rangle+
\langle A^{-1/4}w_{\lambda,+}'(t),A^{3/4}u_{\lambda,+}(t)\rangle
\right).
\end{eqnarray*}

Now estimate the terms in the right-hand side. As in the previous case, from (\ref{est:cv'}) and the bound from below in (\ref{est:bound-cv}) we obtain that
\begin{equation}
c_{v}'(t)\left(|A^{1/2}w_{\lambda,-}(t)|^{2}+|A^{3/4}w_{\lambda,-}(t)|^{2}+|A^{1/4}w_{\lambda,+}(t)|^{2}\right)\leq
\frac{8L_{0}R_{1}^{2}}{\nu_{0}}F_{\lambda}(t).
\nonumber
\end{equation}

Now we observe that
\begin{eqnarray*}
|A^{1/2}u(t)|^{2}-|A^{1/2}v(t)|^{2} & = & 
\langle A^{1/2}w(t),A^{1/2}(u(t)+v(t))\rangle
\\[0.5ex]
& = &
\langle A^{1/2}w_{\lambda,-}(t),A^{1/2}(u_{\lambda,-}(t)+v_{\lambda,-}(t))\rangle
\\[0.5ex]
&  &
\mbox{}+\langle A^{1/4}w_{\lambda,+}(t),A^{3/4}(u_{\lambda,+}(t)+v_{\lambda,+}(t))\rangle.
\end{eqnarray*}

Therefore, from (\ref{hp:R0}), (\ref{hp:R1-u}) and (\ref{hp:R1-v}) we obtain that
\begin{eqnarray*}
\left||A^{1/2}u(t)|^{2}-|A^{1/2}v(t)|^{2}\right| & \leq & 
|A^{1/2}w_{\lambda,-}(t)|\cdot\left(|A^{1/2}u_{\lambda,-}(t)|+|A^{1/2}v_{\lambda,-}(t)|\right)
\\[0.5ex]
& & 
\mbox{}+|A^{1/4}w_{\lambda,+}(t)|\cdot\left(|A^{3/4}u_{\lambda,+}(t)|+|A^{3/4}v_{\lambda,+}(t)|\right)
\\
& \leq &
\frac{1}{\sqrt{\nu_{0}}}\cdot F_{\lambda}(t)^{1/2}\cdot(2R_{0}+3R_{1}),
\end{eqnarray*}
and therefore from (\ref{est:cu-cv}) we deduce that
\begin{equation}
|c_{u}(t)-c_{v}(t)|\leq\frac{L_{0}(2R_{0}+3R_{1})}{\sqrt{\nu_{0}}}F_{\lambda}(t)^{1/2}.
\nonumber
\end{equation}

Finally, from (\ref{hp:R1-u}) we obtain that
\begin{eqnarray*}
|\langle w_{\lambda,-}'(t),Au_{\lambda,-}(t)\rangle| & = &
|\langle A^{1/4}w_{\lambda,-}'(t),A^{3/4}u_{\lambda,-}(t)\rangle|
\\
& \leq &
|A^{1/4}w_{\lambda,-}'(t)|\cdot|A^{3/4}u(t)|
\\
& \leq &
F_{\lambda}(t)^{1/2}\cdot R_{1},
\end{eqnarray*}
and analogously
\begin{equation}
|\langle A^{-1/4}w_{\lambda,+}'(t),A^{3/4}u_{\lambda,+}(t)\rangle|\leq
|A^{-1/4}w_{\lambda,+}'(t)|\cdot|A^{3/4}u_{\lambda,+}(t)|\leq
F_{\lambda}(t)^{1/2}\cdot R_{1},
\nonumber
\end{equation}
while from (\ref{hp:R2-noreg}) we obtain that
\begin{equation}
|\langle A^{1/4}w_{\lambda,-}'(t),A^{5/4}u_{\lambda,-}(t)\rangle|\leq 
|A^{1/4}w_{\lambda,-}'(t)|\cdot|A^{5/4}u_{\lambda,-}(t)|\leq
F_{\lambda}(t)^{1/2}\cdot R_{2,\lambda}.
\nonumber
\end{equation}

From all these estimates we deduce that
\begin{equation}
F_{\lambda}'(t)\leq
\left(\frac{8L_{0}R_{1}^{2}}{\nu_{0}}+\frac{2L_{0}(2R_{0}+3R_{1})(2R_{1}+R_{2,\lambda})}{\sqrt{\nu_{0}}}\right)
F_{\lambda}(t).
\qquad
\forall  t\in[0,S_{0}].
\nonumber
\end{equation}

Integrating this differential inequality we conclude that
\begin{equation}
F_{\lambda}(t)\leq F_{\lambda}(0)\exp\left(\left(
\frac{8L_{0}R_{1}^{2}}{\nu_{0}}+\frac{2L_{0}(2R_{0}+3R_{1})(2R_{1}+R_{2,\lambda})}{\sqrt{\nu_{0}}}
\right)t\right)
\qquad
\forall  t\in[0,S_{0}].
\nonumber
\end{equation}

Keeping (\ref{equiv:F<E}) and (\ref{equiv:F>E}) into account, this inequality implies that
\begin{equation}
E_{w}^{\lambda,-}(t)\leq\Gamma_{1,\lambda}E_{w}(0)
\exp\left(\Gamma_{2,\lambda}t\right),
\label{est:Ew-}
\end{equation}
with $\Gamma_{1,\lambda}$ and $\Gamma_{2,\lambda}$ given by (\ref{defn:G1-noreg}) and (\ref{defn:G2-noreg}). This estimates the contribution of low-frequency components.

In order to estimate the contribution of high-frequency component, we observe that 
\begin{equation}
|w_{\lambda,+}'(t)|^{2}\leq
\left(|u_{\lambda,+}'(t)|+|v_{\lambda,+}'(t)|\right)^{2}\leq
2|u_{\lambda,+}'(t)|^{2}+2|v_{\lambda,+}'(t)|^{2},
\nonumber
\end{equation}
and the same for the other terms in the definition of $E_{w}^{\lambda,+}(t)$. As a consequence, it turns out that
\begin{equation}
E_{w}^{\lambda,+}(t)\leq 2E_{u}^{\lambda,+}(t)+2E_{v}^{\lambda,+}(t).
\label{est:Ew+-}
\end{equation}

In order to estimate $E_{v}^{\lambda,+}(t)$, we observe that $v_{\lambda,+}$ is a solution to
\begin{equation}
v_{\lambda,+}''(t)+c_{v}(t)Av_{\lambda,+}(t)=0,
\nonumber
\end{equation}
and we apply Lemma~\ref{lemma:basic} with
\begin{equation}
z(t):=v_{\lambda,+}(t),
\qquad
c(t):=c_{v}(t),
\qquad
\alpha\in\{0,1/4\}.
\nonumber
\end{equation}

Recalling (\ref{est:cv'}), in this case we obtain that
\begin{equation}
E_{v}^{\lambda,+}(t)\leq 
E_{v}^{\lambda,+}(0)\frac{\max\{1,C_{0}\}}{\min\{1,\nu_{0}\}}
\exp\left(\frac{8L_{0}R_{1}^{2}}{\nu_{0}}t\right).
\nonumber
\end{equation}

An analogous argument applies to $u_{\lambda,+}$, leading to (in this case, since (\ref{hp:R1-u}) is stronger than (\ref{hp:R1-v}), the upper bound for $c_{u}(t)$ is smaller, and this would allow to replace the 8 in the exponential by a 2)
\begin{equation}
E_{u}^{\lambda,+}(t)\leq 
E_{u}^{\lambda,+}(0)\frac{\max\{1,C_{0}\}}{\min\{1,\nu_{0}\}}
\exp\left(\frac{8L_{0}R_{1}^{2}}{\nu_{0}}t\right).
\nonumber
\end{equation}

Plugging the last two estimates into (\ref{est:Ew+-}) we deduce that
\begin{equation}
E_{w}^{\lambda,+}(t)\leq\Gamma_{3}\left(E_{u}^{\lambda,+}(0)+E_{v}^{\lambda,+}(0)\right)\exp(\Gamma_{4}t),
\label{est:Ew+}
\end{equation}
with $\Gamma_{3}$ and $\Gamma_{4}$ given by (\ref{defn:G34}).

At this point (\ref{th:wp-noreg}) follows from (\ref{est:Ew-}) and (\ref{est:Ew+}).
\end{proof}

In the third result of this section we consider a given solution to problem (\ref{K:eqn})--(\ref{K:data}), and we assume that it is defined at least in some interval $[0,T]$. We show that, if $(v_{0},v_{1})$ is close enough to $(u_{0},u_{1})$ in $D(A^{3/4})\times D(A^{1/4})$, then also the solution with initial datum $(v_{0},v_{1})$ is defined at least on the interval $[0,T]$. The key point is that we provide an effective estimate on what ``close enough'' means, depending on the growth of $u(t)$.


\begin{prop}[Quantitative local existence around local solutions]\label{prop:loc-quant}

Let us assume that $H$, $A$, and $m$ satisfy the hypotheses stated at the beginning of section~\ref{sec:statements}.

Let $T>0$, and let $u$ be a solution to problem (\ref{K:eqn})--(\ref{K:data}) in the space (\ref{defn:strong-sol}). Let $M(\sigma)$ be defined by (\ref{defn:M}), let $H_{0}$ and $R_{0}$ be real numbers such that
\begin{equation}
H_{0}=|u_{1}|^{2}+M\left(|A^{1/2}u_{0}|^{2}\right),
\qquad\qquad
R_{0}\geq\left(\frac{2H_{0}}{\nu_{0}}\right)^{1/2},
\label{defn:R0}
\end{equation}
and let us define $C_{0}$ and $L_{0}$ as in (\ref{defn:C0-L0}). Let $R_{1}$ be a positive real number satisfying
\begin{equation}
\max\left\{|A^{1/4}u'(t)|,|A^{3/4}u(t)|\right\}\leq R_{1}
\qquad
\forall t\in[0,T].
\label{hp:R1}
\end{equation}

Let $v$ be another local solution to equation (\ref{K:eqn}), possibly defined on a different time interval, with initial data $(v_{0},v_{1})\in D(A^{3/4})\times D(A^{1/4})$ such that
\begin{equation}
|v_{1}|^{2}+M\left(|A^{1/2}v_{0}|^{2}\right)\leq 2H_{0},
\label{hp:ham-v}
\end{equation}

Then the following statements hold true.
\begin{enumerate}
\renewcommand{\labelenumi}{(\arabic{enumi})}

\item  \emph{(More regular case).} Let us assume in addition that $u\in C^{0}\left([0,T],D(A^{5/4})\right)$, and let $R_{2}$ be a real number such that
\begin{equation}
|A^{5/4}u(t)|\leq R_{2}
\qquad
\forall t\in[0,T].
\label{hp:R2-reg-T}
\end{equation}

Let us assume that
\begin{equation}
E_{(u-v)}(0)\cdot\Gamma_{1}\exp\left(\Gamma_{2}T\right)< R_{1}^{2},
\label{hp:E0-small-reg}
\end{equation}
where $\Gamma_{1}$ and $\Gamma_{2}$ are defined by (\ref{defn:G12-reg}).

Then it turns out that $\LS(v_{0},v_{1})>T$ and the difference $u-v$ satisfies
\begin{equation}
E_{(u-v)}(t)\leq E_{(u-v)}(0)\cdot\Gamma_{1}\exp(\Gamma_{2}t)
\qquad
\forall t\in[0,T].
\label{th:wp-E0-reg}
\end{equation}

\item  \emph{(Minimal regularity case).} Let us assume that $u$ has just the regularity (\ref{defn:strong-sol}). Let $\lambda>0$ be a real number such that
\begin{equation}
E_{u}^{\lambda,+}(0)\cdot\Gamma_{3}\exp\left(\Gamma_{4}T\right)
<\frac{R_{1}^{2}}{6},
\label{hp:lambda-noreg}
\end{equation}
where the constants $\Gamma_{3}$ and $\Gamma_{4}$ are defined by (\ref{defn:G34}), and let $R_{2,\lambda}$ be a positive real number such that
\begin{equation}
|A^{5/4}u_{\lambda,-}(t)|\leq R_{2,\lambda}
\qquad
\forall t\in[0,T].
\label{hp:R2-noreg-T}
\end{equation}

Let us assume that
\begin{equation}
E_{(u-v)}(0)\cdot
\left\{\Gamma_{1,\lambda}\exp\left(\Gamma_{2,\lambda}T\right)+2\Gamma_{3}\exp\left(\Gamma_{4}T\right)\right\}
<\frac{R_{1}^{2}}{2},
\label{hp:E0-small-noreg}
\end{equation}
where $\Gamma_{1,\lambda}$ and $\Gamma_{2,\lambda}$ are defined by (\ref{defn:G1-noreg}) and (\ref{defn:G2-noreg}).

Then it turns out that $\LS(v_{0},v_{1})>T$ and for every $t\in[0,T]$ the difference $u-v$ satisfies
\begin{eqnarray*}
E_{(u-v)}(t) & \leq &
E_{(u-v)}(0)\left\{\Gamma_{1,\lambda}\exp(\Gamma_{2,\lambda}t)+2\Gamma_{3}\exp(\Gamma_{4}t)\right\}
\\[0.5ex]
& & \mbox{}+3E_{u}^{\lambda,+}(0)\cdot\Gamma_{3}\exp(\Gamma_{4}t).
\end{eqnarray*}

\end{enumerate}

\end{prop}

\begin{proof}

By the last statement in Theorem~\ref{thmbibl:loc-ex}, we know that $v(t)$ can be extended to some maximal interval $[0,T_{\mathrm{max}})$, where either $T_{\mathrm{max}}=+\infty$ or (\ref{th:alternative}) holds true for $v(t)$. We claim that $T_{\mathrm{max}}>T$ and actually
\begin{equation}
\max\left\{|A^{1/4}v'(t)|,|A^{3/4}v(t)|\right\}\leq 2R_{1}
\qquad
\forall t\in[0,T].
\label{est:claim-S}
\end{equation}

To this end, we set
\begin{equation}
S_{0}:=\sup\left\{
\tau\in(0,T_{\mathrm{max}}):\max\left\{|A^{1/4}v'(t)|,|A^{3/4}v(t)|\right\}\leq 2R_{1}\quad\forall t\in[0,\tau]
\right\},
\nonumber
\end{equation}
and we claim that $S_{0}\geq T$.

To begin with, we have to show that $S_{0}$ is the supremum of a nonempty set. This is true because $\Gamma_{1}\geq 1$ and $\Gamma_{2}\geq 0$ (and analogously $\Gamma_{1,\lambda}\geq 1$ and $\Gamma_{2,\lambda}\geq 0$), and hence both (\ref{hp:E0-small-reg}) and (\ref{hp:E0-small-noreg}) imply in particular that $E_{(u-v)}(0)<R_{1}^{2}$, and therefore a fortiori
\begin{equation}
|A^{1/4}(u_{1}-v_{1})|<R_{1}
\qquad\text{and}\qquad
|A^{3/4}(u_{0}-v_{0})|<R_{1}.
\nonumber
\end{equation}

From these estimates it follows that
\begin{equation}
|A^{1/4}v'(0)|=
|A^{1/4}v_{1}|\leq
|A^{1/4}u_{1}|+|A^{1/4}(u_{1}-v_{1})|<
2R_{1},
\nonumber
\end{equation}
and analogously
\begin{equation}
|A^{3/4}v(0)|=
|A^{3/4}v_{0}|\leq
|A^{3/4}u_{0}|+|A^{3/4}(u_{0}-v_{0})|<
2R_{1}.
\nonumber
\end{equation}

This means that the inequality in (\ref{est:claim-S}) is strict when $t=0$, and therefore by continuity it remains true at least for small positive times, which proves that $S_{0}$ is well defined.

Now let us assume by contradiction that $S_{0}<T$. Since the inequality in (\ref{est:claim-S}) is true for all $t<S_{0}$, we deduce that $S_{0}<T_{\mathrm{max}}$, because otherwise (\ref{th:alternative}) can not be true for $v(t)$. At this point, from the maximality of $S_{0}$ we conclude that
\begin{equation}
\max\left\{|A^{1/4}v'(S_{0})|,|A^{3/4}v(S_{0})|\right\}= 2R_{1},
\label{est:S-max}
\end{equation}
while of course 
\begin{equation}
\max\left\{|A^{1/4}v'(t)|,|A^{3/4}v(t)|\right\}\leq 2R_{1}
\qquad
\forall t\in[0,S_{0}].
\nonumber
\end{equation}

We claim that the functions $u$ and $v$ satisfy, in the interval $[0,S_{0}]$, the assumptions of Proposition~\ref{prop:quant-wp}. To this end, we already know that (\ref{hp:R1-u}) and (\ref{hp:R1-v}) are true in $[0,S_{0}]$. Let us check that (\ref{hp:R0}) holds true with $R_{0}$ defined by (\ref{defn:R0}). Indeed, from the energy equality of Theorem~\ref{thmbibl:loc-ex} and assumption (\ref{hp:ham-v}) we obtain that
\begin{equation}
|v'(t)|^{2}+M\left(|A^{1/2}v(t)|^{2}\right)=
|v_{1}|^{2}+M\left(|A^{1/2}v_{0}|^{2}\right)\leq 
2H_{0}
\qquad
\forall t\in[0,T_{\mathrm{max}}),
\label{en-eq-v}
\end{equation}
while of course
\begin{equation}
|u'(t)|^{2}+M\left(|A^{1/2}u(t)|^{2}\right)=H_{0}\leq 2H_{0}
\qquad
\forall t\in[0,T].
\label{en-eq-u}
\end{equation}

The strict hyperbolicity assumption (\ref{hp:m-sh}) implies that $M(\sigma)\geq\nu_{0}\sigma$ for every $\sigma\geq 0$, and therefore from (\ref{en-eq-v}) and (\ref{en-eq-u}) we obtain (\ref{hp:R0}), as requested.

Now the proof proceeds in two different ways depending on the regularity of $u(t)$.

\subparagraph{\textmd{\textit{Proof of statement~(1)}}}

In this case $u(t)$ and $v(t)$ satisfy the assumptions of statement~(1) of Proposition~\ref{prop:quant-wp}, from which we deduce that
\begin{equation}
E_{(u-v)}(t)\leq E_{(u-v)}(0)\cdot\Gamma_{1}\exp(\Gamma_{2}t)
\qquad
\forall t\in[0,S_{0}].
\label{est:Ew(t)-S0}
\end{equation}

Since $S_{0}<T$, this inequality with $t=S_{0}$, combined with the smallness assumption (\ref{hp:E0-small-reg}), implies that $E_{(u-v)}(S_{0})<R_{1}^{2}$, and hence a fortiori
\begin{equation}
|A^{1/4}(u'(S_{0})-v'(S_{0}))|<R_{1}
\qquad\text{and}\qquad
|A^{3/4}(u(S_{0})-v(S_{0}))|<R_{1}.
\nonumber
\end{equation}

This in turn implies that
\begin{equation}
|A^{1/4}v'(S_{0})|\leq
|A^{1/4}u'(S_{0})|+|A^{1/4}(u'(S_{0})-v'(S_{0}))|<
2R_{1},
\nonumber
\end{equation}
and analogously
\begin{equation}
|A^{3/4}v(S_{0})|\leq
|A^{3/4}u(S_{0})|+|A^{3/4}(u(S_{0})-v(S_{0}))|<
2R_{1}.
\nonumber
\end{equation}

The last two inequalities contradict (\ref{est:S-max}), which shows that $S_{0}=T$ in this case. As a consequence, $v(t)$ exists at least up to $T$, and satisfies the inequality in (\ref{est:Ew(t)-S0}) in the interval $[0,T]$, which proves (\ref{th:wp-E0-reg}).

\subparagraph{\textmd{\textit{Proof of statement~(2)}}}

In this case we exploit statement~(2) of Proposition~\ref{prop:quant-wp} with the value of $\lambda$ for which (\ref{hp:lambda-noreg}) holds true. To begin with, we observe that
\begin{equation}
|v_{\lambda,+}'(0)|^{2}\leq
2|u_{\lambda,+}'(0)|^{2}+2|u_{\lambda,+}'(0)-v_{\lambda,+}'(0)|^{2}\leq
2|u_{\lambda,+}'(0)|^{2}+2|u'(0)-v'(0)|^{2}.
\nonumber
\end{equation}

Arguing in the same way with the other terms of $E_{v}^{\lambda,+}(0)$ we find that
\begin{equation}
E_{v}^{\lambda,+}(0)\leq 2 E_{u}^{\lambda,+}(0)+2E_{(u-v)}(0).
\nonumber
\end{equation}

Therefore, from (\ref{th:wp-noreg}) we obtain that
\begin{equation}
E_{(u-v)}(t)\leq
E_{(u-v)}(0)\left\{\Gamma_{1,\lambda}\exp(\Gamma_{2,\lambda}t)+2\Gamma_{3}\exp(\Gamma_{4}t)\right\}+
3\Gamma_{3}E_{u}^{\lambda,+}(0)\exp(\Gamma_{4}t)
\nonumber
\end{equation}
for every $t\in[0,S_{0}]$. Setting $t=S_{0}$, and assuming that $S_{0}<T$, from (\ref{hp:lambda-noreg}) and (\ref{hp:E0-small-noreg}) we deduce that $E_{(u-v)}(S_{0})<R_{1}^{2}$. At this point the conclusion follows as in the previous case.
\end{proof}

The last result of this section is an interpolation inequality similar to~\cite[Proposition~3.3]{gg:K-Nishihara}. The idea is the following. It is well-known that one can always estimate an ``intermediate'' quantity in terms of a ``weaker'' and a ``stronger'' quantity. For example, if $\{a_{k}\}$ and $\{\lambda_{k}\}$ are two sequences of nonnegative real numbers, and $0<b<c$ are two real exponents, then it turns out that
\begin{equation}
\sum_{k=1}^{\infty}a_{k}\lambda_{k}^{b}\leq
\left(\sum_{k=1}^{\infty}a_{k}\right)^{1-\theta}
\cdot\left(\sum_{k=1}^{\infty}a_{k}\lambda_{k}^{c}\right)^{\theta}
\qquad\text{with}\qquad
\theta:=\frac{b}{c}.
\nonumber
\end{equation}

We observe that the higher is the exponent $c$, the smaller is the exponent $\theta$ of the corresponding sum in the right-hand side of the inequality.

In the following result we extend this type of inequality to strong quantities such as those defined in (\ref{defn:SG-spaces}), the idea being that the dependence on the strong quantity becomes of logarithmic type when $\varphi$ grows enough at infinity.

\begin{lemma}[Interpolation inequality]\label{lemma:interpolation}

Let $\varphi:[0,+\infty)\to[0,+\infty)$ be an increasing continuous function with $\varphi(0)=0$, and let $b$ be a positive real number such that
\begin{equation}
K_{b}:=\sup\left\{\sigma^{b}\exp\left(-\frac{1}{2}\varphi(\sigma)\right):\sigma\geq 0\right\}<+\infty.
\label{defn:ka}
\end{equation}

Let $\{a_{k}\}$ and $\{\lambda_{k}\}$ be two sequences of nonnegative real numbers such that
\begin{equation}
0<E:=\sum_{k=1}^{\infty}a_{k}<+\infty
\qquad\text{and}\qquad
F:=\sum_{k=1}^{\infty}a_{k}\max\{1,\lambda_{k}\}\exp(\varphi(\lambda_{k}))<+\infty.
\nonumber
\end{equation}

Then it turns out that
\begin{equation}
\sum_{k=1}^{\infty}a_{k}\lambda_{k}^{b}\leq
\left\{K_{b}+\left[\varphi^{-1}\left(2\log\frac{F}{E}\right)\right]^{b}\right\}E,
\label{th:est-ElogF}
\end{equation}
where $\varphi^{-1}$ denotes the inverse function of $\varphi$.

\end{lemma}

\begin{proof}

Let us partition the positive integers into the two subsets
\begin{equation}
A:=\left\{k\geq 1:\lambda_{k}\leq\varphi^{-1}\left(2\log\frac{F}{E}\right)\right\},
\qquad
B:=\left\{k\geq 1:\lambda_{k}>\varphi^{-1}\left(2\log\frac{F}{E}\right)\right\}.
\nonumber
\end{equation}

From the definition of $A$ it follows that
\begin{equation}
\sum_{k\in A}a_{k}\lambda_{k}^{b}\leq
\left[\varphi^{-1}\left(2\log\frac{F}{E}\right)\right]^{b}\sum_{k\in A}a_{k}\leq
\left[\varphi^{-1}\left(2\log\frac{F}{E}\right)\right]^{b}E.
\label{est:sum-A}
\end{equation}

From the definition of $B$, and the monotonicity of $\varphi^{-1}$, it follows that
\begin{equation}
\exp\left(\frac{1}{2}\varphi(\lambda_{k})\right)\geq
\exp\left(\log\frac{F}{E}\right)=
\frac{F}{E}
\qquad
\forall k\in B,
\nonumber
\end{equation}
and therefore by (\ref{defn:ka})
\begin{equation}
\lambda_{k}^{b}=
\lambda_{k}^{b}\exp\left(-\frac{1}{2}\varphi(\lambda_{k})\right)\cdot
\exp\left(-\frac{1}{2}\varphi(\lambda_{k})\right)\cdot
\exp(\varphi(\lambda_{k}))\leq
K_{b}\cdot\frac{E}{F}\cdot\exp(\varphi(\lambda_{k}))
\nonumber
\end{equation}
for every $k\in B$, from which we conclude that
\begin{equation}
\sum_{k\in B}a_{k}\lambda_{k}^{b}\leq
K_{b}\cdot\frac{E}{F}\cdot\sum_{k\in B}a_{k}\exp(\varphi(\lambda_{k}))\leq
K_{b}\cdot\frac{E}{F}\cdot F=
K_{b}E.
\label{est:sum-B}
\end{equation}

Summing (\ref{est:sum-A}) and (\ref{est:sum-B}) we obtain (\ref{th:est-ElogF}).
\end{proof}


\setcounter{equation}{0}
\section{Proof of the main results}\label{sec:proofs}

\subsection{Proof of Theorem~\ref{thm:lsc}}

Let us fix any positive real number $T<\LS(u_{0},u_{1})$, and let us apply Proposition~\ref{prop:loc-quant} in the interval $[0,T]$ with $(v_{0},v_{1}):=(u_{0\ep},u_{1\ep})$. Since $(u_{0\ep},u_{1\ep})\to(u_{0},u_{1})$ in $D(A^{1/2})\times H$, if $H_{0}$ is defined as in (\ref{defn:R0}) then (\ref{hp:ham-v}) is satisfied when $\ep$ is small enough. Now the argument proceeds in a slightly different way depending on the regularity of $u$.

\paragraph{\textmd{\textit{Case with more regularity}}}

If $(u_{0},u_{1})\in D(A^{5/4})\times D(A^{3/4})$, then from the preservation of regularity in Theorem~\ref{thmbibl:loc-ex} we know that $u(t)$ remains bounded in the same space for every $t\in[0,T]$. Therefore, we can choose real numbers $H_{0}$, $R_{0}$, $R_{1}$, and $R_{2}$ in such a way that (\ref{defn:R0}), (\ref{hp:R1}), and (\ref{hp:R2-reg-T}) hold true. Now we choose $\ep_{0}>0$ such that, for every $\ep\in(0,\ep_{0})$, the number $E_{(u-u_{\ep})}(0)$ satisfies (\ref{hp:E0-small-reg}) for every $\ep\in(0,\ep_{0})$. 

At this point statement~(1) of Proposition~\ref{prop:loc-quant} guarantees that $\LS(u_{0\ep},u_{1\ep})\geq T$, and
\begin{equation}
E_{(u-u_{\ep})}(t)\leq E_{(u-u_{\ep})}(0)\cdot\Gamma_{1}\exp(\Gamma_{2}t)
\qquad
\forall t\in[0,T],
\quad
\forall\ep\in(0,\ep_{0}).
\nonumber
\end{equation}

Since $T$ is arbitrary, this proves (\ref{th:LS-lsc}) and (\ref{th:u-uep-reg}) in the more regular case.

\paragraph{\textmd{\textit{Case with minimal regularity}}}

If the function $u$ has just the regularity (\ref{defn:strong-sol}), then we define $H_{0}$, $R_{0}$, $R_{1}$ in such a way that (\ref{defn:R0}) and (\ref{hp:R1}) hold true, and we choose $\lambda$ in such a way that (\ref{hp:lambda-noreg}) is satisfied (here it is essential that the constants $\Gamma_{3}$ and $\Gamma_{4}$ do not depend on $\lambda$). 

Now we choose $R_{2,\lambda}$ so that (\ref{hp:R2-noreg-T}) holds true, and we choose $\ep_{0}>0$ such that, for every $\ep\in(0,\ep_{0})$, the number  $E_{(u-u_{\ep})}(0)$ satisfies (\ref{hp:E0-small-noreg}) for every $\ep\in(0,\ep_{0})$ (in this case the smallness of $\ep_{0}$ depends both on $T$ and on $\lambda$). 

At this point statement~(2) of Proposition~\ref{prop:loc-quant} guarantees that $\LS(u_{0\ep},u_{1\ep})\geq T$, which proves (\ref{th:LS-lsc}) also in this case. In addition, we obtain also that
\begin{equation}
E_{(u-u_{\ep})}(t)\leq
E_{(u-u_{\ep})}(0)\left\{\Gamma_{1,\lambda}\exp(\Gamma_{2,\lambda}t)+2\Gamma_{3}\exp(\Gamma_{4}t)\right\}+
3\Gamma_{3}E_{u}^{\lambda,+}(0)\exp(\Gamma_{4}t)
\nonumber
\end{equation}
for every $t\in[0,T]$ and every $\ep\in(0,\ep_{0})$, and letting $\ep\to 0^{+}$ we deduce that
\begin{equation}
\limsup_{\ep\to 0^{+}}\sup_{t\in[0,T]}E_{(u-u_{\ep})}(t)\leq
3\Gamma_{3}E_{u}^{\lambda,+}(0)\exp(\Gamma_{4}T).
\nonumber
\end{equation}

Since this is true for every $\lambda$ for which (\ref{hp:lambda-noreg}) is satisfied, letting $\lambda\to +\infty$ we obtain exactly (\ref{th:u-uep-noreg}).
\qed



\subsection{Proof of Theorem~\ref{thm:growth}}

If $|u_{1}|^{2}+|A^{1/2}u_{0}|^{2}=0$, then there is nothing to prove because the solution is the constant $u(t)\equiv u_{0}$ with $u_{0}\in\ker A$, and therefore all the estimates are trivial.

Otherwise, let us consider the classical Hamiltonian $H(t)$ defined in (\ref{defn:ham}), the classical $\alpha$-energy (for every real number $\alpha\geq 0$)
\begin{equation}
\mathcal{E}_{\alpha}(t):=|A^{\alpha}u'(t)|^{2}+|A^{\alpha+1/2}u(t)|^{2},
\nonumber
\end{equation}
and the uncorrected $\varphi$-energy
\begin{equation}
\widehat{F}_{\varphi}(t):=\sum_{k=1}^{\infty}\max\{1,\lambda_{k}\}
\left(\langle u'(t),e_{k}\rangle^{2}+\lambda_{k}^{2}\langle u(t),e_{k}\rangle^{2}\right)
\exp(\varphi(\lambda_{k})).
\nonumber
\end{equation}

From the classical energy equality we know that $H(t)$ is equal to a positive constant for every $t\geq 0$. Thanks to the usual coercivity estimate $M(\sigma)\geq\nu_{0}\sigma$ for every $\sigma\geq 0$, this implies a uniform bound on $|u'(t)|$ and $|A^{1/2}u(t)|$, and therefore also an estimate of the form
\begin{equation}
\nu_{0}|A^{1/2}u(t)|^{2}\leq 
M\left(|A^{1/2}u(t)|^{2}\right)\leq
C_{0}|A^{1/2}u(t)|^{2},
\nonumber
\end{equation}
which implies that
\begin{equation}
\frac{1}{\max\{1,C_{0}\}}H(t)\leq
\mathcal{E}_{0}(t)\leq
\frac{1}{\min\{1,\nu_{0}\}}H(t).
\nonumber
\end{equation}

Since $H(t)$ is a positive constant, this means that $\mathcal{E}_{0}(t)$ is bounded from below and from above by two positive constants.

Finally, we observe that the uncorrected $\varphi$-energy is bounded from above by a multiple of the corrected version (\ref{defn:F-phi}), and more precisely
\begin{equation}
\widehat{F}_{\varphi}(t)\leq
\frac{1}{\min\{1,\nu_{0}\}}F_{\varphi}(t)
\qquad
\forall t\geq 0.
\label{est:equiv-Fphi}
\end{equation}

Now we consider the analytic and the quasi-analytic scenario separately.

\paragraph{\textmd{\textit{Analytic data}}}

In the case where $\varphi(\sigma)=r_{0}\sigma$, the differential inequality (\ref{th:diff-ineq-F}) reads as
\begin{equation}
F_{\varphi}'(t)\leq
c_{0}F_{\varphi}(t)\left\{1+\frac{1}{r_{0}}\log\frac{F_{\varphi}(t)}{c_{1}}\right\}
\qquad
\forall t\geq 0.
\label{diff-ineq:an}
\end{equation}

Integrating this differential inequality we deduce that (\ref{th:F-an}) holds true with
\begin{equation}
\beta_{1}:=c_{0}+\frac{c_{0}}{r_{0}}+\frac{c_{0}}{r_{0}}\log\frac{F_{\varphi}(0)}{c_{1}}.
\nonumber
\end{equation}

Indeed, (\ref{diff-ineq:an}) is equivalent to saying that $F_{\varphi}(t)$ is a subsolution of the differential equation
\begin{equation}
y'(t)=c_{0}y(t)\left\{1+\frac{1}{r_{0}}\log\frac{y(t)}{c_{1}}\right\}
\nonumber
\end{equation}
with initial datum $F_{\varphi}(0)$, while a direct substitution reveals that the right-hand side of (\ref{th:F-an}) is a supersolution of the same differential equation with initial datum $F_{\varphi}(0)\cdot e$.

In order to prove (\ref{th:an-est}), for every $\alpha\geq 0$ and every $t\geq 0$ we apply Lemma~\ref{lemma:interpolation} with
\begin{equation}
b:=4\alpha,
\qquad\qquad
a_{k}:=\langle u'(t),e_{k}\rangle^{2}+\lambda_{k}^{2}\langle u(t),e_{k}\rangle^{2}.
\label{defn:ak}
\end{equation}

With these choices it turns out that
\begin{equation}
E:=\mathcal{E}_{0}(t),
\qquad\qquad
F:=\widehat{F}_{\varphi}(t),
\qquad\qquad
\sum_{k=1}^{\infty}a_{k}\lambda_{k}^{b}=\mathcal{E}_{\alpha}(t),
\nonumber
\end{equation}
and therefore inequality (\ref{th:est-ElogF}) reads as
\begin{equation}
|A^{\alpha}u'(t)|^{2}+|A^{\alpha+1/2}u(t)|^{2}\leq 
\left\{K_{4\alpha}+\left[\frac{2}{r_{0}}\log\left(\frac{\widehat{F}_{\varphi}(t)}{\mathcal{E}_{0}(t)}\right)\right]^{4\alpha}\right\}\mathcal{E}_{0}(t).
\nonumber
\end{equation}

Since $\mathcal{E}_{0}(t)$ is bounded from above and from below by positive constants, when we plug (\ref{est:equiv-Fphi}) and (\ref{th:F-an}) into this inequality we obtain (\ref{th:an-est}).
 
\paragraph{\textmd{\textit{Quasi analytic data}}}

Let $\varphi^{-1}:[0,+\infty)\to[0,+\infty)$ denote the inverse of the function $\varphi$ defined in (\ref{defn:qa}). To begin with, we observe that there exists a positive real number $c_{2}$ such that
\begin{equation}
\varphi^{-1}(\sigma)\leq c_{2}\sigma\log(2+\sigma)
\qquad
\forall\sigma\geq 0.
\label{est:inv-phi}
\end{equation}

Indeed, if we set $\psi(\sigma):=\sigma\log(2+\sigma)$, it is enough to observe that $\varphi^{-1}(\sigma)$ and $\psi(\sigma)$ are positive for $\sigma>0$ and
\begin{equation}
\lim_{\sigma\to 0^{+}}\frac{\varphi^{-1}(\sigma)}{\psi(\sigma)}=
\lim_{\sigma\to 0^{+}}\frac{\sigma}{\psi(\varphi(\sigma))}=
1
\qquad\text{and}\qquad
\lim_{\sigma\to +\infty}\frac{\varphi^{-1}(\sigma)}{\psi(\sigma)}=
\lim_{\sigma\to +\infty}\frac{\sigma}{\psi(\varphi(\sigma))}=
1.
\nonumber
\end{equation}

At this point the differential inequality (\ref{th:diff-ineq-F}) implies that
\begin{equation}
F_{\varphi}'(t)\leq
c_{0}F_{\varphi}(t)\left\{1+c_{2}\log\frac{F_{\varphi}(t)}{c_{1}}\cdot
\log\left(2+\log\frac{F_{\varphi}(t)}{c_{1}}\right)\right\}
\qquad
\forall t\geq 0.
\label{diff-ineq:qa}
\end{equation}

Integrating this differential inequality we deduce that (\ref{th:F-qa}) holds true with
\begin{equation}
\beta_{2}:=c_{0}+c_{0}c_{2}\left(1+\log\frac{F_{\varphi}(0)}{c_{1}}\right)
\left(1+\log\left(3+\log\frac{F_{\varphi}(0)}{c_{1}}\right)\right).
\nonumber
\end{equation}

Indeed, (\ref{diff-ineq:qa}) is equivalent to saying that $F_{\varphi}(t)$ is a subsolution of the differential equation
\begin{equation}
y'(t)=c_{0}y(t)\left\{1+c_{2}\log\frac{y(t)}{c_{1}}\cdot\log\left(2+\log\frac{y(t)}{c_{1}}\right)\right\}
\nonumber
\end{equation}
with initial datum $F_{\varphi}(0)$, while a direct substitution reveals that the right-hand side of (\ref{th:F-qa}) is a supersolution of the same differential equation with initial datum $F_{\varphi}(0)\cdot e^{e}$.

In order to prove (\ref{th:qa-est}), for every $\alpha\geq 0$ and every $t\geq 0$ we apply Lemma~\ref{lemma:interpolation} with $b$ and $a_{k}$ defined again as in (\ref{defn:ak}). Keeping (\ref{est:inv-phi}) into account, in this case (\ref{th:est-ElogF}) implies that
\begin{equation}
\mathcal{E}_{\alpha}(t)\leq 
\left\{K_{4\alpha}+c_{2}^{4\alpha}\left[2\log\left(\frac{\widehat{F}_{\varphi}(t)}{\mathcal{E}_{0}(t)}\right)\cdot
\log\left(2+2\log\frac{\widehat{F}_{\varphi}(t)}{\mathcal{E}_{0}(t)}\right)\right]^{4\alpha}\right\}\mathcal{E}_{0}(t).
\nonumber
\end{equation}

Since $\mathcal{E}_{0}(t)$ is bounded from above and from below by positive constants, when we plug (\ref{th:F-qa}) into this inequality we obtain an estimate of the form
\begin{equation}
\mathcal{E}_{\alpha}(t)\leq
c_{3,\alpha}\left[\exp(\beta_{2}t)\exp(\exp(\beta_{2}t))\strut\right]^{4\alpha}
\nonumber
\end{equation}
for a suitable constant $c_{3,\alpha}$, and this estimate in turn implies (\ref{th:qa-est}).
\qed


\subsection{Proof of Theorem~\ref{thm:AGE}}

\paragraph{\textmd{\textit{Finite dimensional data}}}

As in the proof of Theorem~\ref{thm:growth}, from the classical energy equality and the strict hyperbolicity assumption (\ref{hp:m-sh}) we obtain a uniform bound on $|u'(t)|$ and $|A^{1/2}u(t)|$. Since the solution lies in a $A$-invariant subspace of $H$ with finite dimension, this is enough to deduce that $|A^{\alpha}u'(t)|$ and $|A^{\alpha+1/2}u(t)|$ are bounded uniformly in $t\geq 0$ for every real number $\alpha\geq 0$, of course with a bound that depends on $\alpha$ and on the maximal eigenvalue of $A$ in this subspace.

In particular, there exists a constant $B_{0}$ such that
\begin{equation}
\max\left\{|A^{1/4}u'(t)|^{2},|A^{3/4}u(t)|^{2},|A^{5/4}u(t)|^{2}\right\}\leq B_{0}
\qquad
\forall t\geq 0.
\label{th:fd-est}
\end{equation}

Now we claim that, for $\ep>0$ small enough, the assumptions of statement~(1) of Proposition~\ref{prop:loc-quant} are satisfied with $H_{0}$ and $R_{0}$ defined as in (\ref{defn:R0}) (with equality in the definition of $R_{0}$), and
\begin{equation}
(v_{0},v_{1}):=(u_{0\ep},u_{1\ep}),
\qquad\qquad
R_{1}=R_{2}:=B_{0},
\qquad\qquad
T:=\frac{1}{\Gamma_{2}}|\log\ep|,
\nonumber
\end{equation}
where $\Gamma_{2}$ is defined by (\ref{defn:G12-reg}). Indeed, estimates (\ref{hp:R1}) and (\ref{hp:R2-reg-T}) are true because of (\ref{th:fd-est}), estimate (\ref{hp:ham-v}) is true when $\ep$ is small enough, while (\ref{hp:E0-small-reg}) is true because
\begin{equation}
E_{(u-u_{\ep})}(0)\cdot\Gamma_{1}\exp(\Gamma_{2}T)\leq
\ep^{2}\cdot\Gamma_{1}\exp(|\log\ep|)=
\Gamma_{1}\ep\leq
B_{0}^{2}=
R_{1}^{2}.
\nonumber
\end{equation}

\paragraph{\textmd{\textit{Analytic data}}}

Let us apply (\ref{th:an-est}) with $\alpha=1/4$ and $\alpha=3/4$. We deduce that there exists a constant $B_{1}$ such that
\begin{equation}
\max\left\{|A^{1/4}u'(t)|^{2},|A^{3/4}u(t)|^{2},|A^{5/4}u(t)|^{2}\right\}\leq B_{1}\exp(3\beta_{1}t)
\qquad
\forall t\geq 0.
\label{th:an-est-B1}
\end{equation}

We claim that, for $\ep>0$ small enough, the assumptions of statement~(1) of Proposition~\ref{prop:loc-quant} are satisfied with $H_{0}$ and $R_{0}$ defined as in (\ref{defn:R0}) (with equality in the definition of $R_{0}$), and
\begin{equation}
(v_{0},v_{1}):=(u_{0\ep},u_{1\ep}),
\qquad\qquad
R_{1}=R_{2}:=\left(\frac{\nu_{0}}{10L_{0}}\right)^{1/2}|\log\ep|^{1/4},
\label{def:R12-an}
\end{equation}
and
\begin{equation}
T:=\frac{1}{12\beta_{1}}\log(|\log\ep|).
\nonumber
\end{equation}

Since (\ref{hp:ham-v}) is true when $\ep$ is small enough, it remains to check that also inequalities (\ref{hp:R1}), (\ref{hp:R2-reg-T}) and (\ref{hp:E0-small-reg}) are satisfied when $\ep$ is small enough.

\begin{itemize}

\item  As for (\ref{hp:E0-small-reg}), we observe that when $R_{1}$ and $R_{2}$ are given by (\ref{def:R12-an}), then the constant $\Gamma_{2}$ defined in (\ref{defn:G12-reg}) satisfies $\Gamma_{2}\leq|\log\ep|^{1/2}$ when $\ep$ is small enough. At this point for $\ep$ small enough it turns out that
\begin{equation}
E_{(u-u_{\ep})}(0)\cdot\Gamma_{1}\exp(\Gamma_{2}T)\leq
\ep^{2}\cdot\Gamma_{1}\exp(|\log\ep|^{1/2}T)\leq
R_{1}^{2},
\nonumber
\end{equation}
where the last inequality is true because the left-hand side tends to~0 and the right-hand side tends to $+\infty$ as $\ep\to 0^{+}$.

\item As for (\ref{hp:R1}) and (\ref{hp:R2-reg-T}), we need to check that
\begin{equation}
\max\left\{|A^{1/4}u'(t)|^{2},|A^{3/4}u(t)|^{2},|A^{5/4}u(t)|^{2}\right\}\leq
\frac{\nu_{0}}{10L_{0}}|\log\ep|^{1/2}.
\nonumber
\end{equation}

Thanks to (\ref{th:an-est-B1}), it is enough to show that
\begin{equation}
B_{1}\exp\left(\frac{1}{4}\log(|\log\ep|)\right)\leq
\frac{\nu_{0}}{10L_{0}}|\log\ep|^{1/2},
\nonumber
\end{equation}
and this is true when $\ep$ is small enough.

\end{itemize}

\paragraph{\textmd{\textit{Classical quasi-analytic data}}}

As in the analytic case, we apply (\ref{th:qa-est}) with $\alpha=1/4$ and $\alpha=3/4$, and we deduce that there exists a constant $B_{2}$ such that
\begin{equation}
\max\left\{|A^{1/4}u'(t)|^{2},|A^{3/4}u(t)|^{2},|A^{5/4}u(t)|^{2}\right\}\leq B_{2}\exp(\exp(3\beta_{2}t))
\qquad
\forall t\geq 0.
\nonumber
\end{equation}

Now again we claim that, when $\ep>0$ is small enough, the assumptions of statement~(1) of Proposition~\ref{prop:loc-quant} are satisfied with the choices (\ref{def:R12-an}) and 
\begin{equation}
T:=\frac{1}{12\beta_{2}}\log(\log(|\log\ep|)).
\nonumber
\end{equation}

The proof of (\ref{hp:E0-small-reg}) is analogous to the analytic case, because we used only that $\Gamma_{2}T\leq|\log\ep|$ for $\ep$ small. The proof of (\ref{hp:R1}) and (\ref{hp:R2-reg-T}) reduces to the inequality
\begin{equation}
B_{2}\exp\left(\exp\left(\frac{1}{4}\log(\log(|\log\ep|))\right)\right)\leq
\frac{\nu_{0}}{10L_{0}}|\log\ep|^{1/2},
\nonumber
\end{equation}
which again is true when $\ep$ is small enough.
\qed

\begin{rmk}[Back to the null solution]\label{rmk:basic}
\begin{em}

In the case where $u(t)$ is the null solution, this technique leads to the classical result that $\LS(u_{0\ep},u_{1\ep})\geq C/\ep^{2}$. Indeed, it is enough to apply statement~(1) of Proposition~\ref{prop:loc-quant} with
\begin{equation}
(v_{0},v_{1}):=(u_{0\ep},u_{1\ep}),
\qquad\qquad
R_{0}=R_{1}=R_{2}:=a_{0}\ep,
\qquad\qquad
T:=\frac{a_{1}}{\ep^{2}}
\nonumber
\end{equation}
for suitable choices of the constants $a_{0}$ and $a_{1}$.

\end{em}
\end{rmk}








\subsubsection*{\centering Acknowledgments}

Both authors are members of the Italian {\selectlanguage{italian}%
``Gruppo Nazionale per l'Analisi Matematica, la Probabilit\`{a} e le loro Applicazioni'' (GNAMPA) of the ``Istituto Nazionale di Alta Matematica'' (INdAM)}. The first author was partially supported by PRIN 2020XB3EFL, ``Hamiltonian and Dispersive PDEs''.

\selectlanguage{english}



\label{NumeroPagine}

\end{document}